\documentclass[12pt]{amsart}
\usepackage{graphicx}
\usepackage[headings]{fullpage}
\usepackage{amssymb,epic,eepic,epsfig,amsbsy,amsmath,amscd}
\numberwithin{equation}{section}
                        \theoremstyle{plain}
\usepackage{mathrsfs}

\newcommand{\psdraw}[2]
         {\begin{array}{c} \hspace{-1.3mm}
         \raisebox{-4pt}{\psfig{figure=#1.eps,width=#2}}
         \hspace{-1.9mm}\end{array}}
         
\newcommand\no[1]{}

\newtheorem{theorem}{Theorem}[section]
\newtheorem{thm}{Theorem}
\newtheorem{lemma}[theorem]{Lemma}

\newtheorem{claim}[theorem]{Claim}

\theoremstyle{definition}
\newtheorem{remark}[theorem]{Remark}

\def\BC{\mathbb C}

\def\BZ{\mathbb Z}
\def\BR{\mathbb R}

\def\fb{\mathfrak b}

\def\la{\langle}
\def\ra{\rangle}

\DeclareMathOperator{\tr}{\mathrm tr}

\def\ve{\varepsilon}
\def\be { \begin{equation} }
\def\ee { \end{equation} }

\begin{document}

\title[Twisted Alexander polynomials and $SL_2(\BC)$-representations]
{On the twisted Alexander polynomial for representations into $SL_2(\BC)$}

\author[Anh T. Tran]{Anh T. Tran}
\address{Department of Mathematics, The Ohio State University, Columbus, OH 43210, USA}
\email{tran.350@osu.edu}


\begin{abstract}
We study the twisted Alexander polynomial $\Delta_{K,\rho}$ of a knot $K$ associated to a non-abelian representation $\rho$ of the knot group into $SL_2(\BC)$. It is known for every knot $K$ that if $K$ is fibered, then for every non-abelian representation, $\Delta_{K,\rho}$ is monic and has degree $4g(K)-2$ where $g(K)$ is the genus of $K$. Kim and Morifuji recently proved the converse for $2$-bridge knots. In fact they proved a stronger result: if a $2$-bridge knot $K$ is non-fibered, then all but finitely many non-abelian representations on some component have $\Delta_{K,\rho}$ non-monic and degree $4g(K)-2$. In this paper, we consider two special families of non-fibered $2$-bridge knots including twist knots. For these families, we calculate the number of non-abelian representations where $\Delta_{K,\rho}$ is monic and calculate the number of non-abelian representations where the degree of $\Delta_{K,\rho}$ is less than $4g(K)-2$. 
\end{abstract}

\maketitle

\section{Main results}

The twisted Alexander polynomial was introduced by Lin \cite{Lin01-1} 
for knots in $S^3$ and by Wada \cite{Wada94-1} for finitely presentable groups. 
It is a generalization of the classical Alexander polynomial and gives a powerful 
tool in low dimensional topology. One of the most important applications 
is the determination of fiberedness \cite{FV11-1} and genus (the Thurston norm) \cite{FV12-1} 
of knots by the collection of the twisted Alexander
polynomials corresponding to all finite-dimensional representations. 
For literature on other applications and related topics, we refer to the survey paper by
Friedl and Vidussi \cite{FV10-1}.

Following \cite{KM}, we study the twisted Alexander polynomial $\Delta_{K,\rho}$ of a knot $K$ associated to a representation $\rho$ of the knot group into $SL_2(\BC)$. See Section \ref{tap} for a definiton of $\Delta_{K,\rho}$. In general, $\Delta_{K,\rho}$ is a rational function in one variable. However if the representation $\rho$ is non-abelian, then $\Delta_{K,\rho}$ is a Laurent polynomial and has degree less than or equal to $4g(K)-2$ where $g(K)$ is the genus of $K$. The twisted Alexander polynomials of two conjugate representations are the same, hence we can always consider representations up to conjugation. It is known for every knot $K$ that if $K$ is fibered, then for every non-abelian representation, $\Delta_{K,\rho}$ is monic \cite{GKM05-1} and has degree $4g(K)-2$ \cite{KM05-1}. Kim and Morifuji \cite{KM} recently proved the converse for $2$-bridge knots. In fact they proved a stronger result: if a $2$-bridge knot $K$ is non-fibered, then all but finitely many non-abelian  representations on some component have $\Delta_{K,\rho}$ non-monic and degree $4g(K)-2$. 

In this paper, we look into two special families of non-fibered $2$-bridge knots, the knots $\fb(6n+1,3)$ and the twist knots. For these two families, we calculate the number of non-abelian representations where $\Delta_{K,\rho}$ is monic and calculate the number of non-abelian representations where the degree of $\Delta_{K,\rho}$ is less than $4g(K)-2$.

Let $\fb(p,m)$ be the $2$-bridge knot associated to a pair of relatively prime odd integers $p>m>1$ (see e.g. \cite{BZ}), and $K_m~(m>0)$ the $m$-twist knot (see Figures \ref{even} and \ref{odd} in Section 4). Then we have the following.

\begin{thm}
For $K=\fb(6n+1,3)$, the number of (conjugacy classes of) non-abelian representations where the degree of $\Delta_{K,\rho}$ is less than $4g(K)-2$ is equal to $2n$, and the number of non-abelian representations where $\Delta_{K,\rho}$ is monic is also equal to $2n$.
\label{m=3}
\end{thm}

\begin{thm}
(i) For $K=K_{2n}$, the number of non-abelian representations where the degree of $\Delta_{K,\rho}$ is less than $4g(K)-2$ is equal to $1+(-1)^n$, and the number of non-abelian representations where $\Delta_{K,\rho}$ is monic is equal to $2n-2-a_n-b_n$, where
\begin{eqnarray*}
a_n &=&  
\begin{cases}
2, & n \equiv 1 \pmod 6 \\
0, & \text{\emph{otherwise}}
\end{cases}, \\
b_n &=& \begin{cases}
2, & n \equiv 1 \pmod{5} \\
0, & \text{\emph{otherwise}}
\end{cases}.
\end{eqnarray*}

(ii) For $K=K_{2n-1}$, the number of non-abelian representations where the degree of $\Delta_{K,\rho}$ is less than $4g(K)-2$ is equal to $1+(-1)^n$, and the number of non-abelian representations where $\Delta_{K,\rho}$ is monic is equal to $2n-2-c_n-d_n-e_n$, where
\begin{eqnarray*}
 c_n &=&  
\begin{cases}
1, & n \equiv 1 \pmod 3 \\
0, & \text{\emph{otherwise}}
\end{cases},\\
d_n &=&  
\begin{cases}
2, & n \equiv 5 \pmod 6 \\
0, & \text{\emph{otherwise}}
\end{cases},\\
e_n &=& \begin{cases}
2, & n \equiv 4 \pmod{5} \\
0, & \text{\emph{otherwise}}
\end{cases}.
\end{eqnarray*}
\label{m=p-2}
\end{thm}

The paper is organized as follows. We review the definition of the twisted Alexander polynomial and some related work on fibering and genus of knots in Section \ref{tap}. We then prove Theorem \ref{m=3} in Section 3 and Theorem \ref{m=p-2} in Section 4.

We would like the referee for comments and suggestions that greatly improve the exposition of the paper. We would also like to thank T. Morifuji for comments.

\section{Twisted Alexander polynomials} 

\label{tap}

Let $K$ be a knot in $S^3$ and $G_K=\pi_1(S^3\backslash K)$ its knot group. We choose and fix a Wirtinger presentation 
$$
G_K=
\langle a_1,\ldots,a_\ell~|~r_1,\ldots,r_{\ell-1}\rangle.
$$
Then 
the abelianization homomorphism 
$f:G_K\to H_1(S^3\backslash K,\BZ)
\cong {\BZ}
=\langle t
\rangle$ 
is given by 
$f(a_1)=\cdots=f(a_\ell)=t$. 
Here 
we specify a generator $t$ of $H_1(S^3\backslash K;\BZ)$ 
and denote the sum in $\BZ$ multiplicatively. 
Let us 
consider a linear representation 
$\rho:G_K\to SL_2(\BC)$. 

These maps naturally induce two ring homomorphisms 
$\widetilde{\rho}: {\BZ}[G_K] \rightarrow M(2,{\BC})$ 
and $\widetilde{f}:{\BZ}[G_K]\rightarrow {\BZ}[t^{\pm1}]$, 
where ${\BZ}[G_K]$ is the group ring of $G_K$ 
and 
$M(2,{\BC})$ is the matrix algebra of degree $2$ over ${\BC}$. 
Then 
$\widetilde{\rho}\otimes\widetilde{f}$ 
defines a ring homomorphism 
${\BZ}[G_K]\to M\left(2,{\BC}[t^{\pm1}]\right)$. 
Let 
$F_\ell$ denote the free group on 
generators $a_1,\ldots,a_\ell$ and 
$\Phi:{\BZ}[F_\ell]\to M\left(2,{\BC}[t^{\pm1}]\right)$
the composition of the surjection 
${\BZ}[F_q]\to{\BZ}[G_K]$ 
induced by the presentation of $G_K$ 
and the map 
$\widetilde{\rho}\otimes\widetilde{f}:{\BZ}[G_K]\to M(2,{\BC}[t^{\pm1}])$. 

Let us consider the $(\ell-1)\times \ell$ matrix $M$ 
whose $(i,j)$-component is the $2\times 2$ matrix 
$$
\Phi\left(\frac{\partial r_i}{\partial a_j}\right)
\in M\left(2,{\BZ}[t^{\pm1}]\right),
$$
where 
${\partial}/{\partial a}$ 
denotes the free differential calculus. 
For 
$1\leq j\leq \ell$, 
let us denote by $M_j$ 
the $(\ell-1)\times(\ell-1)$ matrix obtained from $M$ 
by removing the $j$th column. 
We regard $M_j$ as 
a $2(\ell-1)\times 2(\ell-1)$ matrix with coefficients in 
${\BC}[t^{\pm1}]$. 
Then Wada's twisted Alexander polynomial 
of a knot $K$ 
associated to a representation $\rho:G_K\to SL_2({\BC})$ 
is defined to be a rational function 
$$
\Delta_{K,\rho}(t)
=\frac{\det M_j}{\det\Phi(1-a_j)} 
$$
and 
moreover well-defined 
up to a factor $t^{2k}~(k\in{\BZ})$, see \cite{Wada94-1}. It is known that if two
representations $\rho, \rho'$ are conjugate, then $\Delta_{K,\rho}=\Delta_{K,\rho'}$ holds. Hence, in this paper we can always consider representations up to conjugation.

A representation $\rho:G_K\to SL_2(\BC)$ is called non-abelian if $\rho(G_K)$ is a non-abelian subgroup of $SL_2(\BC)$. Suppose $\rho$ is a non-abelian representation. It is known that the twisted Alexander polynomial $\Delta_{K,\rho}(t)$  is always a Laurent polynomial for any knot $K$ \cite{KM05-1}, and is a monic polynomial if $K$ is a fibered knot. It is also known that the converse holds for $2$-bridge knots \cite{KM}. Moreover,  if $K$ is a knot of genus $g(K)$ then $\deg(\Delta_{K,\rho}(t))\leq 4g(K)-2$ \cite{FK06-1}, and the equality holds if $K$ is fibered \cite{KM05-1}. 

We say that the twisted Alexander polynomial $\Delta_{K,\rho}(t)$ determines the fiberedness of $K$ if it is a monic polynomial if and only if $K$ is fibered. We also say $\Delta_{K,\rho}(t)$ determines the knot genus $g(K)$ if $\deg(\Delta_{K,\rho}(t))=4g(K)-2$ holds. 

\section{Proof of Theorem \ref{m=3}}

The following result is well-known. 

\begin{lemma}
The knot $\fb(p,3)$ is non-fibered if and only if $p \equiv 1 \pmod 6$.
\end{lemma}

\begin{proof}
The proof is based on the calculation of the Alexander polynomial. Since $\fb(p,3)$ is an alternating knot, it is fibered if and only if its Alexander polynomial is monic, i.e. has leading coefficient 1, see \cite{Cr, Mu}. The leading coefficient of the Alexander polynomial of $\fb(p,3)$ is 2 if $p \equiv 1 \pmod 6$ and is 1 if $p \equiv -1 \pmod 6$, see e.g. \cite[Section 2]{HS}. The lemma follows.
\end{proof}

The standard presentation of the knot group of $K=\fb(p,m)$ is $G_K=\la a,b \mid wa=bw\ra$ where $w=a^{\ve_1}b^{\ve_2} \dots a^{\ve_{p-2}}b^{\ve_{p-1}}$ and 
$\ve_j=(-1)^{\lfloor jq/p \rfloor}$, see e.g. \cite{BZ}. Here $a$ and $b$ are 2 standard generators of a 2-bridge knot group.

 For a representation $\rho: G_K \to SL_2(\BC)$, let $x=\tr \rho(a)=\tr \rho(b)$ and $z=\tr \rho(ab)$.  Let $d=\frac{p-1}{2}$. By \cite{Le}, the representation $\rho$ is non-abelian if and only if $R_w(x,z)=0$ where
$$
R_w(x,z)=\tr \rho(w)-\tr \rho(w')+\dots+(-1)^{d-1}\tr \rho(w^{(d-1)})+(-1)^d.
$$ 
Here if $u$ is a word in the letters $a,b$, then $u'$ denotes the word obtained from $u$ 
by deleting the two letters at the two ends. 
Then $w^{(d-1)}$ denotes the element obtained from $w$ by applying the deleting operation 
$d-1$ times. 

Let $\{S_j(z)\}_j$ be the sequence of Chebyshev polynomials defined by $S_0(z)=1,\,S_1(z)=z$, and $S_{j+1}(z)=zS_j(z)-S_{j-1}(z)$ for all integers $j$. By \cite[Proposition 3.1]{NT}, we have the following description of  non-abelian representations of $\fb(6n+1,3)$.

\begin{lemma}
For the 2-bridge knot $\fb(6n+1,3)$, one has $w=(ab)^n(a^{-1}b^{-1})^n(ab)^n$ and
$$R_w(x,z)=S_{3n}(z)-S_{3n-1}(z)-x^2(z-2)S^2_{n-1}(z)(S_n(z)-S_{n-1}(z)).$$
\label{rep}
\end{lemma}

Let $r=waw^{-1}b^{-1}$. Then $\frac{\partial r}{\partial a}=w \left( 1+(1-a)w^{-1}\frac{\partial w}{\partial a} \right)$ where 
$$\frac{\partial w}{\partial a}=\left( 1+(ab)^n(a^{-1}b^{-1})^n \right) \left( 1+ \cdots + (ab)^{n-1}\right)-(ab)^n \left( 1 + \cdots+ (a^{-1}b^{-1})^{n-1} \right) a^{-1}.$$
Suppose $\rho: G_K \to SL_2(\BC)$ is a non-abelian representation. Then the twisted Alexander polynomial of $K$ associated to $\rho$ is 
$$\Delta_{K,\rho}(t)=\det \Phi \left( \frac{\partial r}{\partial a} \right) \big/\det \Phi(1-b)=\det \Phi \left( \frac{\partial r}{\partial a} \right) \big/ (1-tx+t^2).$$ 
We have \begin{eqnarray*}
\det \Phi \left( \frac{\partial w}{\partial a} \right) &=& t^{4n} \mid I+ (I-tA)t^{-2n}(AB)^{-n}(BA)^n(AB)^{-n} \\
&& \big[ \left( I+(AB)^n(A^{-1}B^{-1})^n \right) \left( I+ \cdots + (t^2AB)^{n-1}\right)\\
&& -(t^2AB)^n \left( I + \cdots+ (t^{-2}A^{-1}B^{-1})^{n-1} \right) t^{-1}A^{-1} \big] \mid 
\end{eqnarray*}
where $A=\rho(a)$ and $B=\rho(b)$. Here $|M|$ denotes the determinant of $M$. 

\medskip

We will need the following lemma whose proof is elementary.

\begin{lemma}
Let $L(t)=\sum_{j=r}^{s} M_j \, t^j$ be a Laurent polynomial in $t$ whose coefficients are $2 \times 2$ complex matrices. Suppose $|M_r|$ and $|M_s|$ are non-zero. Then the highest and lowest degree terms of $|L(t)|$ (in $t$)  are  $|M_s| \, t^{2s}$ and $|M_r| \, t^{2r}$ respectively. 
\label{base}
\end{lemma}

The following lemma follows easily from Lemma \ref{base}.

\begin{lemma}
For the two-bridge knot $K=\fb(6n + 1, 3)$, the highest and lowest degree terms of $\det \Phi \left( \frac{\partial w}{\partial a} \right)$ are  $\mid I + A(AB)^{-n}(BA)^n A^{-1} \mid t^{4n}$ and $\mid  I + (AB)^{n}(BA)^{-n} \mid t^{0}$ respectively, provided that $\mid I + A(AB)^{-n}(BA)^n A^{-1}\mid$ and $\mid I + (AB)^{n}(BA)^{-n}\mid$ are non-zero. 
\label{hl}
\end{lemma}

The following two lemmas are standard, see e.g. \cite{MT}.

\begin{lemma}
One has $S^2_j(z)-zS_j(z)S_{j-1}(z)+S^2_{j-1}(z)=1$.
\label{S}
\end{lemma}

\begin{lemma}
\label{recurrence}
Suppose the sequence $\{M_j\}_{j \in \BZ}$ of $2 \times 2$ matrices satisfies 
the recurrence relation $M_{j+1}=zM_j-M_{j-1}$ for all integers $j$. Then 
$
M_j = S_{i-1}(z)M_1-S_{i-2}(z)M_0.
$
\end{lemma}

\begin{lemma}
One has 
$$\mid I + A(AB)^{-n}(BA)^n A^{-1} \mid = \mid  I + (AB)^{n}(BA)^{-n} \mid = 4+(z-2)(z+2-x^2)S^2_{n-1}(z).$$
\label{hl'}
\end{lemma}

\begin{proof}
It is easy to see that $\mid I + A(AB)^{-n}(BA)^n A^{-1} \mid = \mid  I + (AB)^{n}(BA)^{-n} \mid$. By applying Lemma \ref{recurrence} twice, once with $M_j=(AB)^{j}$ and
once with $M_j=(A^{-1}B^{-1})^{j}$, we have 
\begin{eqnarray*}
\tr (AB)^{n}(BA)^{-n} &=& \tr (AB)^{n}(A^{-1}B^{-1})^{n}\\
&=& S^2_{n-1}(z) \tr ABA^{-1}B^{-1}+ S^2_{n-1}(z) \tr I \\
&& - \, S_{n-1}(z)S_{n-2}(z)(\tr AB + \tr A^{-1}B^{-1}).
\end{eqnarray*}

From the identity $\tr CD=\tr C \tr D - \tr CD^{-1}$ for all matrices $C,D$ in $SL_2(\BC)$, it is easy to see that $\tr ABA^{-1}B^{-1}=z^2-zx^2+2x^2-2=2+(z-2)(z+2-x^2)$. Hence, by Lemma \ref{S}, we obtain
\begin{eqnarray*}
\tr (AB)^{n}(BA)^{-n} &=& (2+(z-2)(z+2-x^2))S^2_{n-1}(z) +2S^2_{n-1}(z)-2zS_{n-1}(z)S_{n-2}(z)\\
&=& 2+(z-2)(z+2-x^2)S^2_{n-1}(z).
\end{eqnarray*} 
The lemma follows, since $\det(I+C)=1+\det C+\tr C$.
\end{proof}

\subsubsection{Genus} 
We first note that the genus of $K=\fb(6n+1,3)$ is $n$. Lemmas \ref{rep}, \ref{hl} and \ref{hl'} imply that the number of (conjugacy classes of) non-abelian representations where the degree of $\Delta_{K,\rho}$ is less than $4g(K)-2$  is equal to the number of solutions $(x,z) \in \BC^2$ of the following system 
\begin{eqnarray}
4+(z-2)(z+2-x^2)S^2_{n-1}(z) &=& 0, \label{T}\\
S_{3n}(z)-S_{3n-1}(z)-x^2(z-2)S^2_{n-1}(z)(S_n(z)-S_{n-1}(z)) &=& 0. \label{R}
\end{eqnarray}

Suppose eq. \eqref{T} holds. Then $x^2(z-2)S^2_{n-1}(z)=4+(z^2-4)S^2_{n-1}(z)$. Eq. \eqref{R} is then equivalent to the following
\begin{equation}
S_{3n}(z)-S_{3n-1}(z)-(4+(z^2-4)S^2_{n-1}(z))(S_n(z)-S_{n-1}(z))=0.
\label{rep-equiv}
\end{equation}

\begin{claim}
The equation \eqref{rep-equiv} has $n$ distinct roots, all of which are different from $\pm 2$.
\label{claim}
\end{claim}

\begin{proof}
Let $h_1(z)$ denote the left hand side of eq. \eqref{rep-equiv}. Since $S_j(2)=j+1$ and $S_j(-2)=(-1)^j(j+1)$ for all integers $j$, we have $h_1(2)=-3$ and $h_1(-2)=(-1)^{n+1}(2n+3)$.

Suppose $z \not= \pm 2$. We may write $z=\alpha+\alpha^{-1}$ where $\alpha \not= \pm 1$. Since $S_j(z)=\frac{\alpha^{j+1}-\alpha^{-(j+1)}}{\alpha-\alpha^{-1}}$ for all integers $j$, we have $h_1(z)=\alpha^{-n}\frac{(2\alpha+1)\alpha^{2n}+\alpha+2}{\alpha+1}$. It follows that $h_1(z)$ is a polynomial in $z$ of degree $n$. We want to show that $h_1(z)$ does not have multiple roots. This holds true if we can show that the polynomial $h_2(\alpha)=2\alpha^{2n+1}+\alpha^{2n}+\alpha+2$ does not have multiple roots $\alpha$ such that $|\alpha| \ge 1$. (Note that if $\alpha$ is a root of $h_2$, then so is $\alpha^{-1}$).

Suppose $h_2$ has a multiple root $\alpha \in \BC$ such that $|\alpha| \ge 1$. Then $h_2(\alpha)=h_2'(\alpha)=0$, or equivalently 
$$(2\alpha+1)\alpha^{2n}+\alpha+2=2((2n+1)\alpha+n)\alpha^{2n-1}+1=0.$$
It follows that $\frac{\alpha+2}{2\alpha+1}=\frac{\alpha}{2((2n+1)\alpha+n)}$, i.e. $\alpha^2+(\frac{5}{2}+\frac{3}{4n})\alpha+1=0$. Hence $\alpha+\alpha^{-1}=-(\frac{5}{2}+\frac{3}{4n})<-\frac{5}{2}$. Then we must have $\alpha \in \BR$ and $\alpha<-2$, since $|\alpha| \ge 1$. Hence $h_1(\alpha)=(2\alpha+1)\alpha^{2n}+\alpha+2<0$, a contradiction. The claim follows.
\end{proof}

For each solution $z$ of eq. \eqref{rep-equiv}, we write $z=\alpha+\alpha^{-1}$ for some $\alpha \in \BC^*$. Then from the proof of Claim \eqref{claim}, we have $\alpha \not= \pm 1$ and $\alpha^{-n}\frac{(2\alpha+1)\alpha^{2n}+\alpha+2}{\alpha+1}=0$. It follows that $\alpha^{2n}=-\frac{\alpha+2}{2\alpha+1} \not= 1$. Eq. \eqref{T} is then equivalent to the following $$x^2=\frac{4+(z^2-4)S^2_{n-1}(z)}{(z-2)S^2_{n-1}(z)}=\frac{(\alpha+1)^2(\alpha^{2n}+1)^2}{\alpha(\alpha^{2n}-1)^2}=\frac{(\alpha-1)^2}{9\alpha},$$
i.e. $9x^2=z-2$. Hence Claim \ref{claim} implies that the system \eqref{T}+\eqref{R} has exactly $2n$ complex solutions $(x,z)$. This means that the number of non-abelian representations where the degree of $\Delta_{K,\rho}$ is less than $4g(K)-2$  is equal to $2n$.

\subsubsection{Fiberedness} The number of non-abelian representations where $\Delta_{K,\rho}$ is monic is equal to the number of solutions $(x,z) \in \BC^2$ of the following system
\begin{eqnarray}
4+(z-2)(z+2-x^2)S^2_{n-1}(z) &=& 1, \label{T1}\\
S_{3n}(z)-S_{3n-1}(z)-x^2(z-2)S^2_{n-1}(z)(S_n(z)-S_{n-1}(z)) &=& 0.  \label{repn}
\end{eqnarray}

Suppose eq. \eqref{T1} holds. Then $x^2(z-2)S^2_{n-1}(z)=3+(z^2-4)S^2_{n-1}(z)$, and eq. \eqref{repn} is equivalent to the following 
\begin{equation}
S_{3n}(z)-S_{3n-1}(z)-(3+(z^2-4)S^2_{n-1}(z))(S_n(z)-S_{n-1}(z))=0.
\label{eqn}
\end{equation}

\begin{claim}
The equation \eqref{eqn} has $n$ distinct roots, all of which are different from $\pm 2$.
\label{claim'}
\end{claim}

\begin{proof}
Let $h_3(z)$ denote the left hand side of eq. \eqref{eqn}. It is easy to see that $h_3(2)=-2$ and $h_3(-2)=2(-1)^{n+1}$. Suppose $z \not= \pm 2$. We may write $z=\alpha+\alpha^{-1}$ where $\alpha \not= \pm 1$.
Then $h_3(z)=\alpha^{-n}(\alpha^{2n}+1)$, and the claim follows easily. 
\end{proof}

For each solution $z$ of eq. \eqref{eqn}, we write $z=\alpha+\alpha^{-1}$ for some $\alpha \in \BC^*$. Then from the proof of Claim \eqref{claim'}, we have $\alpha \not= \pm 1$ and $\alpha^{2n}=-1$. Eq. \eqref{T1} is then equivalent to the following $$x^2=\frac{3+(z^2-4)S^2_{n-1}(z)}{(z-2)S^2_{n-1}(z)}=\frac{(\alpha+1)^2(\alpha^{4n}+\alpha^{2n}+1)}{\alpha(\alpha^{2n}-1)^2}=\frac{(\alpha+1)^2}{4\alpha},$$
i.e. $4x^2=z+2$. Hence Claim \ref{claim'} implies that the system \eqref{T1}+\eqref{repn} has exactly $2n$ complex solutions $(x,z)$. This means that the number of non-abelian representations where $\Delta_{K,\rho}$ is monic is equal to $2n$.

This completes the proof of Theorem \ref{m=3}.

\begin{remark}
One can easily show that all the solutions $(x,z)$ of the systems \eqref{T}+\eqref{R} and \eqref{T1}+\eqref{repn} satisfying the condition $x \not= 2$. It follows that the twisted Alexander polynomial associated to any parabolic $SL_2(\BC)$-representation of the knot group of the 2-bridge knot $\fb(6n+1,3)$ determines the fiberedness and genus of $\fb(6n+1,3)$. This gives a proof of a conjecture of Dunfield, Friedl and Jackson \cite{DFJ} for $\fb(6n+1,3)$. 
\end{remark}

\section{Proof of Theorem \ref{m=p-2}}

We first note that the $m$-twist knot $K_m,\,m>0$, is non-fibered if and only if $m>2$. ($K_1$ is the trefoil knot, and $K_2$ is the figure-8 knot.)

\subsection{$K_{2n},\,n>1$} From \cite{Tr-twist} (and also \cite{NT}), the knot group of $K=K_{2n}$ is $\la a,b \mid wa=bw \ra$ where $a,b$ are meridians depicted in Figure \ref{even} and $w=(ba^{-1})^nb(ab^{-1})^n$. Note that this is not the standard presentation of a $2$-bridge knot.

\begin{figure}[htpb]
$$ \psdraw{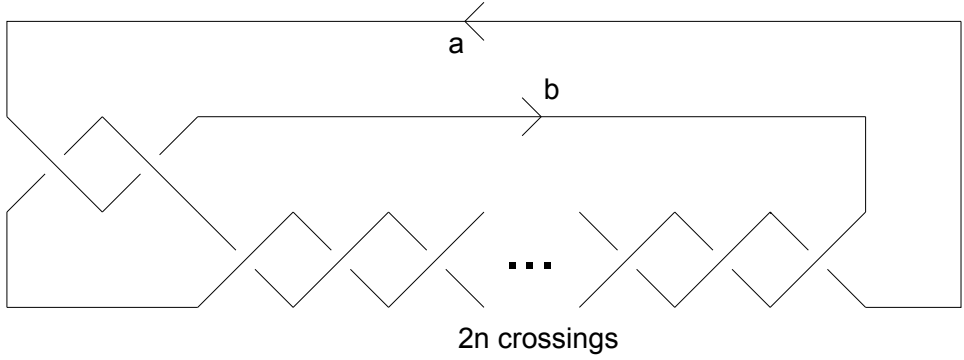}{5in} $$
\caption{$K_{2n}$, $n>0$.}
\label{even}
\end{figure}

For a representation $\rho: G_K \to SL_2(\BC)$, let $x=\tr \rho(a)=\tr \rho(b)$ and $y=\tr \rho(ab^{-1})$. From \cite{Tr-twist}, we have the following
\begin{lemma}
For the twist knot $K = K_{2n}$, a representation $\rho: G_K \to SL_2(\BC)$ is non-abelian if and only if $(x,y) \in \BC^2$ is a root of the polynomial 
$$R_{even}(x,y) = (y+1)S^2_{n-1}(y)-S^2_{n}(y)-2S_{n-1}(y)S_{n}(y)+x^2S_{n-1}(y) (S_{n}(y)-S_{n-1}(y)).$$
\label{rep-even}
\end{lemma}

Let $r=waw^{-1}b^{-1}$. Then $\frac{\partial r}{\partial a}=w \left( 1+(1-a)w^{-1}\frac{\partial w}{\partial a} \right)$ where 
$$\frac{\partial w}{\partial a}=-(1+\cdots+(ba^{-1})^{n-1})ba^{-1}+(ba^{-1})^nb(1+\cdots+(ab^{-1})^{n-1}).$$
Suppose $\rho: G_K \to SL_2(\BC)$ is a non-abelian representation. Then the twisted Alexander polynomial of $K$ associated to $\rho$ is 
$$\Delta_{K,\rho}(t)=\det \Phi \left( \frac{\partial r}{\partial a} \right) \big/\det \Phi(1-b)=\det \Phi \left( \frac{\partial r}{\partial a} \right) \big/ (1-tx+t^2).$$ 
We have 
\begin{eqnarray*}
\det \Phi \left( \frac{\partial w}{\partial a} \right) &=& t^{2}| I+ (I-tA) t^{-1}(BA^{-1})^nB^{-1}(AB^{-1})^n \\
&& \big[ -(1+\cdots+(BA^{-1})^{n-1})BA^{-1}+(BA^{-1})^ntB(1+\cdots+(AB^{-1})^{n-1}) \big] |.
\end{eqnarray*}
where $A=\rho(a)$ and $B=\rho(b)$. 

Let $\{T_j(z)\}_j$ be the sequence of Chebyshev polynomials defined by $T_0(z)=2,\,T_1(z)=z$, and $T_{j+1}(z)=zT_j(z)-T_{j-1}(z)$ for all integers $j$. 

The following lemma follows easily from Lemma \ref{base}.

\begin{lemma}
For the twist knot $K = K_{2n}$, the highest and lowest degree terms of $\det \Phi \left( \frac{\partial w}{\partial a} \right)$ are $|1+\cdots+(AB^{-1})^{n-1}|t^4=\frac{T_n(y)-2}{y-2} \, t^4$ and $|-(1+\cdots+(BA^{-1})^{n-1})|t^0=\frac{T_n(y)-2}{y-2} \, t^0$ respectively, provided that $\frac{T_n(y)-2}{y-2}$ is non-zero. 
\label{hl-even}
\end{lemma}

\subsubsection{Genus} We first note that the genus of $K_{2n}$ is $1$. Lemmas \ref{rep-even} and \ref{hl-even} imply that the number of non-abelian representations where the degree of $\Delta_{K,\rho}$ is less than $4g(K)-2$ is equal to the number of solutions $(x,y) \in \BC^2$ of the following system 
\begin{eqnarray}
(T_n(y)-2)/(y-2) &=& 0, \label{T-even}\\
(y+1)S^2_{n-1}(y)-S^2_{n}(y)-2S_{n-1}(y)S_{n}(y)+x^2S_{n-1}(y) (S_{n}(y)-S_{n-1}(y)) &=& 0. \label{R-even}
\end{eqnarray}

Suppose eq. \eqref{T-even} holds. Then $y \not=2$ and $T_n(y)=2$.
We write $y=\beta+\beta^{-1}$ where $\beta \not=1$. Then $T_n(y)=\beta^{n}+\beta^{-n}=2$ is equivalent to $\beta^n=1$, or equivalent $\beta=e^{i\frac{2\pi  k}{n}}$ for some $1 \le k \le \frac{n}{2}$. It follows that eq. \eqref{T-even} has $\lfloor n/2 \rfloor$ distinct solutions given by $y=2 \cos \frac{2k\pi}{n}$ where $1 \le k \le \frac{n}{2}$. 

If $y=2 \cos \frac{2k\pi}{n}$ for some $1 \le k < \frac{n}{2}$ then $y=\beta+\beta^{-1}$ where $\beta=e^{i\frac{2k\pi}{n}}$. It follows that $S_{n-1}(y)=0$ and $S_{n}(y)=1$. Hence $R_{\emph{even}}$, the left hand side of eq. \eqref{R-even}, is equal to $-1$. 

If $y=-2$ (in this case $n$ must be even), it is easy to see that $R_{\emph{even}}=-(2n+1)(nx^2+2n+1)$, since $S_j(-2)=(-1)^j(j+1)$ for all integers $j$. Then eq. \eqref{R-even} is equivalent to $x^2=-(2+\frac{1}{n})$. 

Hence the system \eqref{T-even}+\eqref{R-even} has exactly $1+(-1)^n$ solutions.

\subsubsection{Fiberedness} The number of non-abelian representations where $\Delta_{K,\rho}$ is monic is equal to the number of solutions $(x,y) \in \BC^2$ of the following system 
\begin{eqnarray}
(T_n(y)-2)/(y-2) &=& 1, \label{T'-even}\\
(y+1)S^2_{n-1}(y)-S^2_{n}(y)-2S_{n-1}(y)S_{n}(y)+x^2S_{n-1}(y) (S_{n}(y)-S_{n-1}(y)) &=& 0. \label{R'-even}
\end{eqnarray}

Suppose eq. \eqref{T'-even} holds. Then $y \not=2$ and $T_n(y)=y$. We write $y=\beta+\beta^{-1}$ where $\beta \not=1$. Then $T_n(y)=y$ is equivalent to $\beta^{n+1}=1$ or $\beta^{n-1}=1$. It follows that the distinct solutions of eq. \eqref{T'-even} are $y=2 \cos \frac{2k\pi}{n+1}$ where $1 \le k < \frac{n+1}{2}$, $y=2 \cos \frac{2k\pi}{n-1}$ where $1 \le k < \frac{n-1}{2}$, and (if $n$ is odd) $y=-2$.

If $y=-2$ (in this case $n$ must be odd), it is easy to see that eq. \eqref{R'-even} is equivalent to $x^2=-(2+\frac{1}{n})$. 

Suppose $y=2 \cos \frac{2k\pi}{n+1}$ where $1 \le k < \frac{n+1}{2}$. Then $y=\beta+\beta^{-1}$ where $\beta=e^{i\frac{2k\pi}{n+1}}$. It follows that $S_{n-1}(y)=-1$ and $S_{n}(y)=0$. Hence $R_{even}=-(x^2+1)$, and eq. \eqref{R'-even} is equivalent to $x^2=-1$.

Suppose $y=2 \cos \frac{2k\pi}{n-1}$ where $1 \le k < \frac{n-1}{2}$. Then $y=\beta+\beta^{-1}$ where $\beta=e^{i\frac{2k\pi}{n-1}}$. It follows that $S_{n-1}(y)=1$ and $S_{n}(y)=y$. Hence $R_{even}=-(y^2+y-1)+x^2(y-1)$. If $y=1$,  i.e. $k=\frac{n-1}{6}$, then $R_{even}=-1$. If $y \not= 1$ then eq. \eqref{R'-even} is equivalent to $x^2=\frac{y^2+y-1}{y-1}$. If $y=-\frac{1+\sqrt{5}}{4}$ (i.e. $k=\frac{2(n-1)}{5}$) or $y=\frac{-1+\sqrt{5}}{4}$ (i.e. $k=\frac{n-1}{5}$) then $y^2+y-1=0$, and eq. \eqref{R'-even} is equivalent to $x=0$.


Hence the system \eqref{T'-even}+\eqref{R'-even} has exactly $2n-2-a_n-b_n$ solutions, where $a_n,b_n$ are as defined in Theorem \ref{m=p-2}(i).

\subsection{$K_{2n-1},\, n>1$} From \cite{Tr-twist} (and also \cite{NT}), the knot group of $K=K_{2n-1}$ is $\la a,b \mid wa=bw \ra$ where $a,b$ are meridians depicted in Figure \ref{odd} and $w=(ab^{-1})^nb(ba^{-1})^n$. Note that this is not the standard presentation of a $2$-bridge knot.

\begin{figure}[htpb]
$$ \psdraw{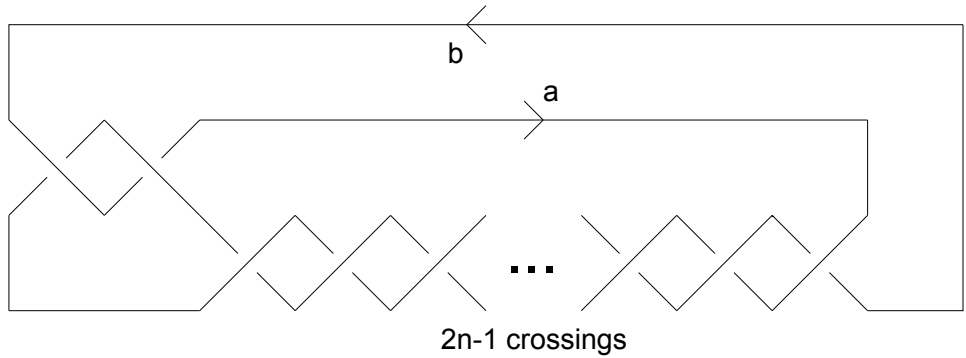}{5in} $$
\caption{$K_{2n-1}$, $n>0$.}
\label{odd}
\end{figure}

For a representation $\rho: G_K \to SL_2(\BC)$, let $x=\tr \rho(a)=\tr \rho(b)$ and $y=\tr \rho(ab^{-1})$. From \cite{Tr-twist}, we have the following
\begin{lemma}
For the twist knot $K = K_{2n-1}$, a representation $\rho: G_K \to SL_2(\BC)$ is non-abelian if and only if $(x,y) \in \BC^2$ is a root of the polynomial 
$$R_{odd}(x,y) = -(y+1)S^2_{n-1}(y)+S^2_{n-2}(y)+2S_{n-1}(y)S_{n-2}(y)+x^2S_{n-1}(y)(S_{n-1}(y)-S_{n-2}(y)).$$
\label{rep-odd}
\end{lemma}

Let $r=waw^{-1}b^{-1}$. Then $\frac{\partial r}{\partial a}=w \left( 1+(1-a)w^{-1}\frac{\partial w}{\partial a} \right)$ where 
$$\frac{\partial w}{\partial a}=(1+\cdots+(ab^{-1})^{n-1})-(ab^{-1})^nb(1+\cdots+(ba^{-1})^{n-1})ba^{-1}.$$
Suppose $\rho: G_K \to SL_2(\BC)$ is a non-abelian representation. Then the twisted Alexander polynomial of $K$ associated to $\rho$ is 
$$\Delta_{K,\rho}(t)=\det \Phi \left( \frac{\partial r}{\partial a} \right) \big/\det \Phi(1-b)=\det \Phi \left( \frac{\partial r}{\partial a} \right) \big/ (1-tx+t^2).$$ 
We have 
\begin{eqnarray*}
\det \Phi \left( \frac{\partial w}{\partial a} \right) &=& t^{2}| I+ (I-tA) t^{-1}(AB^{-1})^nB^{-1}(BA^{-1})^n \\
&& \big[ (1+\cdots+(AB^{-1})^{n-1})-(AB^{-1})^ntB(1+\cdots+(BA^{-1})^{n-1})BA^{-1} \big] |.
\end{eqnarray*}
where $A=\rho(a)$ and $B=\rho(b)$. The following lemma follows easily from Lemma \ref{base}.

\begin{lemma}
For the twist knot $K = K_{2n-1}$, the highest and lowest degree terms of $\det \Phi \left( \frac{\partial w}{\partial a} \right)$ are  $|1+\cdots+(BA^{-1})^{n-1}|t^4=\frac{T_n(y)-2}{y-2} \, t^4$ and $|-(1+\cdots+(AB^{-1})^{n-1})|t^0=\frac{T_n(y)-2}{y-2} \, t^0$ respectively, provided that $\frac{T_n(y)-2}{y-2}$ is non-zero. 
\label{hl-odd}
\end{lemma}

Applying Lemmas \ref{rep-odd} and \ref{hl-odd}, the proof of Theorem \ref{m=p-2}(ii) is similar to that of Theorem \ref{m=p-2}(i). This completes the proof of Theorem \ref{m=p-2}.

\begin{remark}
From the proof of Theorem \ref{m=p-2}, one can easily see that the twisted Alexander polynomial associated to any parabolic $SL_2(\BC)$-representation of the knot group of the $m$-twist knot $K_m,\, m>0$, determines the fiberedness and genus of $K_m$. This gives another proof of \cite[Theorem 1.1]{MT} and hence of a conjecture of Dunfield, Friedl and Jackson \cite{DFJ} for $K_m$. 
\end{remark}

\end{document}